\documentclass[twoside,a4paper,reqno,12pt]{amsart}

\usepackage{cite}
\usepackage{mathrsfs}
\usepackage{amsmath}
\usepackage{amsthm}
\usepackage{latexsym}
\usepackage{amssymb,bm}
\usepackage{amsfonts}
\usepackage{longtable}
\usepackage{multirow}
\usepackage[colorlinks,linkcolor=cyan,anchorcolor=blue,citecolor=red,backref=page]{hyperref}
\usepackage{braket}
\usepackage{amsmath,enumerate}
\usepackage[top=28mm,right=25mm,bottom=28mm,left=25mm]{geometry}

\usepackage{graphics}

\def\l{\langle}
\def\r{\rangle}

\usepackage{color}

\def\Aut{{\sf Aut}} \def\Cos{{\sf Cos}}
\def\A{{\rm A}} \def\S{\rm S}
\def\PSL{{\rm PSL}}

\def\a{\alpha}
\def\b{\beta}

\def\N{{\bf N}}

\newtheorem{thm}{Theorem}[section]

\newtheorem{theorem}[thm]{Theorem}
\newtheorem{lemma}[thm]{Lemma}
\newtheorem{coro}[thm]{Corollary}

\newtheorem{problem}[thm]{Problem}

\theoremstyle{definition}

\newtheorem*{remark*}{Remarks}
\newtheorem{example}[thm]{Example}

\newtheorem{construction}[thm]{Construction}

\def\qed{\nopagebreak\hfill{\rule{4pt}{7pt}}
	\medbreak}

\def\K{{\bf K}}

\def\bbZ{{\mathbb{Z}}}
\def\S{{\mathrm{S}}}
\def\A{{\mathrm{A}}}
\def\K{{\mathrm{K}}}
\def\Sym{{\mathrm{Sym}}}

\def\PSU{{\mathrm{PSU}}}
\def\Cos{{\mathrm{Cos}}}

\def\leq{\leqslant}

\def\geq{\geqslant}

\def\calA{{\mathcal A}}
\def\calB{{\mathcal{B}}}

\begin{document}

	\title[Non-solvable covers of the complete graphs]{A construction of $2$-arc-transitive non-solvable covers of complete graphs}

	\author{Jiyong Chen}
	\author{Cai Heng Li}
	\author{Ci Xuan Wu}
    \author{Yan Zhou Zhu}
	
	\address{J. Chen, School of Mathematical Sciences\\
		Xiamen University \\
		Xiamen 361005\\
		P. R. China}
	\email{chenjy1988@xmu.edu.cn}
	
	\address{C.H. Li, SUSTech International Center for mathematics\\
		Department of Mathematics\\
		Southern University of Science and Technology \\
		Shenzhen 518055\\
		P. R. China}
	\email{lich@sustech.edu.cn}
	
	\address{C.X. Wu, School of Statistics and Mathematics \\
		Yunnan University of Finance and Economics\\
		Kunming 650021\\
		P. R. China}
	\email{wucixuan@gmail.com}
	
	\address{Y.Z. Zhu, School of Mathematical Sciences\\
		Xiamen University \\
		Xiamen 361005\\
		P. R. China}
	\email{zhuyz@xmu.edu.cn}
	
	\date\today

\begin{abstract}
We construct connected $2$-arc-transitive covers of complete graphs with non-abelian characteristically simple transformation groups. This solves the existence problem for non-solvable $2$-arc-transitive covers of complete graphs.

\end{abstract}
\maketitle

\section{Introduction}

Throughout this paper, a group is assumed to be finite, and a graph is assumed to be finite, simple and undirected.
Given a graph $\Gamma=(V,E)$, an \textit{$s$-arc} of $\Gamma$ is a sequence of $s+1$ vertices, where any two consecutive vertices are adjacent and any three consecutive vertices are distinct.
Let $\Gamma$ be a connected regular graph.
For a group $G$ of automorphisms of $\Gamma$, the graph $\Gamma$ is called \textit{$(G,s)$-arc-transitive} (or simply \textit{$s$-arc-transitive}) if $G$ acts transitively on the set of all $s$-arcs of $\Gamma$.
Finite $s$-arc-transitive graphs form an important class of graphs and have received considerable attention.
In particular, Weiss~\cite{Weiss} proved, using local action analysis, that there exists no $8$-arc-transitive graph except for cycles.

For a $(G,s)$-arc-transitive graph $\Gamma=(V,E)$, the orbits $\calB$ of a normal subgroup $N$ on $V$ form a $G$-invariant partition of $V$.
The \textit{normal quotient} $\Gamma_N$ of $\Gamma$ is defined as the graph with vertex set $\calB$, where two vertices $B_1$ and $B_2$ are adjacent if and only if there exist $\beta_1\in B_1$ and $\beta_2\in B_2$ that are adjacent in $\Gamma$.
It was shown by Praeger~\cite[Theorem 4.1]{Praeger1} that if $\Gamma$ is a connected $(G,s)$-arc-transitive graph with $s\geqslant 2$ and $N$ has at least $3$ orbits on $V$, then $\Gamma_N$ is $(G/N,s)$-arc-transitive, and $\Gamma$ and $\Gamma_N$ have the same valency.
In this case, $\Gamma$ is called a \textit{cover} of $\Gamma_N$.
To emphasize the normal subgroup $N$, the graph $\Gamma$ is called an \textit{$N$-cover} of $\Gamma_N$, and $N$ is called the \textit{transformation group} of this cover.
In other words, the class of $s$-arc-transitive graphs is closed under taking normal quotients.
The minimal objects in the class are called \textit{basic} $s$-arc-transitive graphs, which are $(G,s)$-arc-transitive graphs $(V,E)$ with no non-trivial normal quotient, so that either $G$ is \textit{quasiprimitive on $V$} (each minimal normal subgroup is transitive on $V$), or $G$ is \textit{bi-quasiprimitive on $V$} (each minimal normal subgroup of $G$ has exactly two orbits on $V$).
Thus studying $s$-arc-transitive graphs can be divided into two steps as follows:
\begin{enumerate}
	\item[(A)] Determine basic $s$-arc-transitive graphs.
	
	\item[(B)] Construct and characterize normal covers of the basic graphs in Step (A).
\end{enumerate}

Praeger gave systematic descriptions of basic $2$-arc-transitive graphs, see~\cite{Praeger1} for non-bipartite graphs and~\cite{Praeger_bipart} for bipartite graphs.
This initiated a program for the study of $s$-arc-transitive graphs, and the program has been successfully applied to characterize some important classes of $s$-arc-transitive graphs, see~\cite{Giudici-King,Xia-Cay,Lu1,ZhouJX} and references therein.

For Step~(B) of the program, a remarkable result due to J.H.~Conway shows that each connected $s$-arc-transitive graph has a non-trivial connected $s$-arc-transitive cover with an elementary abelian $2$-group as its transformation group, see~\cite[Chapter 19]{Biggs-book}.
However, it is generally hard to construct all $s$-arc-transitive covers of given $s$-arc-transitive graphs.
This paper focuses on the `most special' class of graphs.

The automorphism group of the complete graph $\K_n$ is $\S_n$, and $\K_n$ is $(\S_n,2)$-arc-transitive when $n\geqslant 3$.
Concerning the program mentioned above, a natural problem arises:

\begin{problem}
	Construct and characterize $2$-arc-transitive covers of complete graphs.
\end{problem}

There have been some results on the problem, see~\cite{du1998arctransitive, du20052arctransitive,XDKX}, for which the transformation groups are metacyclic or abelian of rank at most $3$.
This paper presents a construction of $2$-arc-transitive covers of the complete graphs with transformation group being non-abelian characteristically simple groups.

\begin{theorem}\label{thm:cover-complete}
    Let $\Sigma=\K_n$ be the complete graph of order $n\geq 4$.
    Then, for any finite non-abelian simple group $T$, there exist connected $(X,2)$-arc-transitive covers of $\Sigma$, where $X=T^d{.}\S_n$ for some integer $d$, and $T^d$ is minimal normal in $X$.
\end{theorem}

\begin{remark*}
    \
    \begin{enumerate}[\rm(1)]
        \item
        If $G$ is isomorphic to a transitive subgroup of $\S_k$, then we say $G$ has a \textit{transitive permutation degree} $k$.
        In other words, a transitive permutation degree of $G$ is equal to $|G:H|$ for some subgroup $H$ of $G$.
        The value of $d$ in the group $X=T^d.\S_n$ is equal to a transitive permutation degree of the group $\S_n$.
        By Lemma~\ref{lem:simplecomplete},  $d\mid(n-1)!$ and  
        \[\frac{1}{2} \binom{n}{\lfloor \frac{n}{2}\rfloor}\leqslant d\leqslant n!,\ \mbox{for $n\geqslant 7$}.\]
        In particular, $d\geqslant 18>1$ when $n\geqslant 7$, and $d$ can be $1$ in the case $n=4$ (see Theorem~\ref{thm:k4}).
        
        \item For any $3$-transitive subgroup $G$ of $\S_n$, a $(T^d.\S_n,2)$-arc-transitive cover of $\K_n$ is also a $(T^d.G,2)$-arc-transitive graph;
        however, $T^d$ may not be a minimal normal subgroup of $T^d.G$.

        \item In ~\cite{Petersen_Cover}, a construction is given of $2$-arc-transitive covers of the Petersen graph with non-solvable transformation groups, which (to the best of our knowledge) were the only known explicit examples of $2$-arc-transitive covers with non-solvable transformation groups before those in Theorem~\ref{thm:cover-complete} were discovered.
    \end{enumerate}
\end{remark*}

Determining the exact value of $d$ appearing in the main theorem is a difficult problem. 
To illustrate the difficulty, we explore the case $n=4$ in Section~\ref{sec:n4}, which leads to the following corollary.  

\begin{coro}\label{coro:k4}
    Let $Y$ and $\Gamma$ be the group and the graph defined in Construction~\ref{con:K4}.
    Then $\Gamma$ is a connected $(Y,2)$-arc-transitive graph with $Y\cong T^d.\S_4$ and $\Gamma$ is a $T^d$-cover of $\K_4$, where $d\in\{1,3,6\}$.
\end{coro}

\section{The construction}\label{sec:completecover}

First, we introduce the coset graph representation for arc-transitive graphs.
Let $\Gamma=(V,E)$ be a $G$-arc-transitive graph.
Fix an arc $(\a,\b)$, and an element $g\in G$ which interchanges $\a$ and $\b$, so that $g\in\N_G(G_{\a\b})$.
Since $G$ is transitive on the vertices, $V$ can be identified with the set of right cosets $[G:G_\a]=\{G_\a x\mid x\in G\}$ such that the action of $G$ on $V$ is equivalent to the coset action, for any $y\in G$,
\[y:\ G_\a x\mapsto G_\a xy,\ \mbox{where $x\in G$}.\]
Identify $G_\a$ with $\a$ and $G_\a g$ with $\b$, so that $\Gamma(\a)$, the set of neighbors of $\alpha$ in $\Gamma$, is $\{G_\a gh\mid h\in G_\a\}$.
It follows that the adjacency relation is determined by
\[G_\a x\sim G_\a y\Longleftrightarrow yx^{-1}\in G_\a gG_\a.\]
Then, in this way, $\Gamma$ is denoted by $\Cos(G,G_\a, G_\a gG_\a)$, called a coset graph.

Conversely, arc-transitive graphs can be constructed as coset graphs from abstract groups.
Let $G$ be a group, $H$ a subgroup of $G$ and $g\in G$.
Define a \textit{coset graph} $\Gamma=\Cos(G,H,HgH)$, which has vertex set $[G:H]$ such that $(Hx, Hy)$ is an arc if and only if $yx^{-1}\in HgH$.
Then $\Gamma$ is $G$-arc-transitive and the following statements hold.
\begin{enumerate}[\rm(1)]
    \item $\Gamma$ is connected if and only if $\braket{H,g}=G$;
    \item $\Gamma$ is undirected if and only if $HgH=Hg^{-1}H $;
    \item $\Gamma$ is $(G,2)$-arc-transitive if and only if $H$ is $2$-transitive on $[H:H\cap H^g]$.
\end{enumerate} 

Note that $\Aut(\K_n)\cong \S_n$ is $2$-arc-transitive on the complete graph $\K_n$, and a subgroup $G\leqslant \Aut(\K_n)$ is $2$-arc-transitive if and only if $G$ is $3$-transitive on the vertices.
Thus, we can express the complete graphs as below.

\begin{example}\label{ex:complete-graph}
Let $\Gamma=(V,E)=\K_n$ for a positive integer $n\geqslant 3$.
Given a $3$-transitive subgroup $G\leqslant\Sym(V)\cong \S_n$ and two points $\a,\b\in V$, the stabilizer $G_\a$ is $2$-transitive on $[G_\a:G_{\a\b}]$.
Let $g\in G$ be such that $(\a,\b)^g=(\b,\a)$.
Then $\Gamma=\Cos(G,G_\a,G_\a gG_\a)$.
\qed
\end{example}

Now, we are ready to construct non-solvable $2$-arc-transitive normal covers for $2$-arc-transitive complete graphs, and prove Theorem~\ref{thm:cover-complete}.

Let $\Omega=\{1,2,\dots,n\}$ with $n\geqslant 4$, and let $G=\Sym(\Omega)=\S_n$.
Consider the $n$-cycle $(12\dots n)\in G$ and let
\[\mbox{$\calA=(12\dots n)^G=\{(12\dots n)^g\mid g\in G\}$}\]
be the set of all $n$-cycles in $G$.
Note that any element of $\calA$ can be uniquely written as $(1i_2i_3\dots i_n)$. Hence, the set $\calA$ can be partitioned into $n-1$ disjoint subsets $O_1,\dots, O_{n-1}$ such that $O_k$ consists of elements with $i_{k+1}=2$, in other words,
\[O_k=\{(1i_2i_3\dots i_n) \in \calA \mid i_{k+1}=2\} =\{\alpha\in \mathcal{A}\mid 1^{\alpha^k}=2\}.\]

Let $T$ be a finite non-abelian simple group. 
By~\cite{2p-generated}, there exist $x,y\in T$ with $|x|=2$ and $|y|$ an odd prime such that $T=\langle x,y\rangle$.
We use the definition of \textit{wreath product} given in~\cite[page 46]{snmaximal}, and we explain some necessary notation in the following context.
Let $F=\mathrm{Fun}(\calA,T)$ be the set consisting of all functions from $\mathcal{A}$ to $T$.
Note that $F$ is a group by defining the multiplication as follows:
\[(f_1f_2)(\alpha)=f_1(\alpha)f_2(\alpha),\mbox{ for all $f_1,f_2\in F$ and $\alpha\in \calA$}.\]
Then $F$ is isomorphic to $T^{|\calA|}$ via the isomorphism $f\mapsto (f(\alpha_1),...,f(\alpha_{|\calA|}))$, where $\calA=\{\alpha_1,...,\alpha_{|\calA|}\}$.
The action of $G$ on $F$ is defined by
\[(f^g)(\alpha)=f(\alpha^{g^{-1}}),\mbox{ for all $f\in F$, $\alpha\in \calA$ and $g\in G$.}\]
Then we have the semidirect product $X=F{:}G$ defined by this action.
In~\cite[page 46]{snmaximal}, this semidirect product is used to define \textit{wreath product} of $T$ by $G$, denoted by $X=T\wr_{\calA}G$.
Hence, we have that 
\[X=F{:}G\cong T\wr_\mathcal{A}G\cong T^{|\mathcal{A}|}{:}\S_n=T^{(n-1)!}{:}\S_n.\]

Let $L=\Sym\{3,\dots,n\}\cong \S_{n-2}$ and let $\delta=(12)$. Consider an element $$\alpha =(1 i_2\dots i_{k} 2 i_{k+2}\dots i_n)\in O_k.$$ Then, for any $\pi \in L$, we have \begin{align*}
    \alpha^\pi& =(1i_2^\pi\dots i_{k}^\pi 2 i_{k+2}^\pi\dots i_n^\pi)\in O_{k}, \\
    \alpha^\delta& =(2i_2\dots i_{k} 1 i_{k+2}\dots i_n)=(1 i_{k+2}\dots i_n2i_2\dots i_{k}  )\in O_{n-k},
\end{align*}
from which the following lemma follows immediately.

\begin{lemma}\label{lem:orbitL}
For $1\leq k\leq n-1$, the group $L$  is regular on $O_k$, and $O_k^{\delta}=O_{n-k}$. 
In particular, $O_1,\dots,O_{n-1}$ are the all orbits of $L$ on $\calA$.
\end{lemma}

Let $H=\Sym\{2,\dots, n\}\cong \S_{n-1}$. Now we are ready to state our construction of $2$-arc-transitive covers of complete graphs.
\begin{construction}\label{cons:simplecomplete}
    Using the definitions above, we choose an element $f\in F$
    such that
    \[ f(\alpha)=\begin{cases}
        y, & \mbox{ if $\alpha\in O_1$;}\\
        y^{-1}, & \mbox{ if $\alpha\in O_{n-1}$;}\\
        x, & \mbox{ if $\alpha\in O_2\cup O_{n-2}$;}\\
        1,&\mbox{ otherwise.}
    \end{cases}\]
    Let $g=f\delta$ with $\delta=(12)\in G$, and let $Y=\langle H,g\rangle$.
    Define $\Gamma=\Cos(Y,H,HgH)$.\qed
\end{construction}

Before proving that the graph $\Gamma$ constructed above is a desired graph, we give a simple lemma about subgroups of $R=T_1\times\cdots \times T_k\cong T^k$.
Let $\rho_i$ be the projection map from $R$ to $T_i$ for each $i$, that is, 
\[\rho_i(t_1,...,t_k)=t_i\mbox{ for $(t_1,...,t_k)\in R=T_1\times\cdots\times T_k.$}\]
A subgroup $S\leqslant R$ is called a \textit{subdirect product} of $R$ if $\rho_i(S)=T_i$ for each $i$.
Furthermore, we call a subdirect product $S$ of $R\cong T^n$ a \textit{full-diagonal} subgroup of $R$ if $S\cong T$.

\begin{lemma}\label{lem:subdirect}
	Let $T$ be a non-abelian simple group.
    Let $X=(T_1\times\cdots \times T_n){:}\S_n\cong T\wr \S_n$, and let $N$ be the minimal normal subgroup of $X$.
	Suppose that $X_0\leqslant X$ acts transitively by conjugation on the set $\{T_1,\dots, T_n\}$ (or equivalently, on the index set $\{1,\dots,n \}$).
	If $X_0\cap N$ projects onto each $T_i$, then the following statements hold.
	\begin{enumerate}[\rm(1)]
		\item  There exists a $X_0$-invariant partition $\{I_1,\dots,I_m\}$ of the set $\{1,2,\dots,n\}$ such that $X_0\cap N=\prod_{j=1}^{m}D_j$, where $D_j\cong T$ is a full-diagonal subgroup of $\prod_{i\in I_j}T_i$;
		\item $X_0\cap N$ is a minimal normal subgroup of $X_0$ and $m\mid n$.
	\end{enumerate}
\end{lemma}
\begin{proof}
    (1).
    Since $X_0\cap N$ projects onto each $T_i$, it follows that $X_0\cap N$ is a subdirect product of $N$.
    Then there exists a partition $\{I_1,\dots,I_m\}$ of $\{1,...,n\}$ such that $X_0\cap N=\prod_{j=1}^{m}D_j$, where $D_j\cong T$ is a full-diagonal subgroup of $\prod_{i\in I_j}T_i$, see~\cite[page 328]{scott1980representations}.
    Since the $D_1,\dots,D_m$ are exactly the minimal normal subgroups of $X_0\cap N$, it follows that for any $x\in X_0$, $x$ induces a permutation on $\{I_1,\dots,I_m\}$, which is given by
    \[
    D_i^x = D_j
    \quad\Longleftrightarrow\quad
    I_i^x = I_j.
    \]
    Noting that the kernel of the action of $X_0$ on $\{I_1,\dots, I_m\}$ contains $X_0\cap N$, we deduce that
    $\{I_1,\dots,I_m\}$ is $X_0$-invariant.

    (2). Since $\overline{X_0}=X_0/(X_0\cap N)$ is a transitive subgroup of $\S_n$, it follows that $X_0$ is transitive on $\{I_1,\dots, I_m\}$, and hence $|I_1|=\cdots=|I_m|$.
    This implies that $m$ is a divisor $n$.
    Consequently, $X_0$ acts transitively by conjugation on $\{D_1,\dots, D_m\}$, and hence $X_0\cap N$ is minimal in $X_0$.
\end{proof}

We now prove Theorem~\ref{thm:cover-complete}.
For convenience, we restate the theorem as the following lemma.

\begin{lemma}\label{lem:simplecomplete}
    Let $\Gamma$ be as constructed in Construction~$\ref{cons:simplecomplete}$.
    Then $Y\cong T^m.\S_n$ for some positive integer $m$ which is a divisor of $(n-1)!$, and $\Gamma$ is a connected $2$-arc-transitive graph, and is a minimal normal cover of $\K_n$.
    Moreover, $m\geqslant\frac{1}{2}\binom{n}{\lfloor\frac{n}{2}\rfloor}$ if $n\geqslant 7$.
\end{lemma}
\begin{proof}
    We first claim that $g^2 = 1$ and $[L, g] = 1$.
    Let $c = g^2 = (f\delta)^2 = f f^\delta$. For any point $\alpha \in \mathcal{A}$, we have $c(\alpha) = f(\alpha) f(\alpha^\delta)$.
    Note that $O_k^\delta = O_{n-k}$. By the definition of $f$, we have $f(\alpha^\delta) = f(\alpha)^{-1}$. Hence,
    $c(\alpha) = f(\alpha) f(\alpha)^{-1} = 1$, and thus 
    \[g^2=c = 1.\]
    Let $z \in L$. Then $f^z(\alpha) = f(\alpha^{z^{-1}})$.
    Recalling that $\alpha$ and $\alpha^{z^{-1}}$ lie in the same $L$-orbit, we have $f^z = f$.
    Since $\delta^z = \delta$, it follows that $g$ centralizes $z$. So $[L, g] = 1$, as claimed.

    Recall $H=\Sym\{2,\dots, n\}$ and $F\cong T^{|\calA|}\lhd X$.
    Since $[L,g]=1$, we have $H\cap H^g\geqslant L$.
    Note that $L$ is maximal in $H$, and $g\notin \N_X(H)$.
    This implies that $H\cap H^g=L$. Therefore, a vertex stabilizer and an arc stabilizer of $\Gamma$ is isomorphic to $\S_{n-1}$ and $\S_{n-2}$, respectively.
    It follows that $\Gamma$ is $2$-arc-transitive and has valency $n-1$.
    Set $M=Y\cap F$.
    Then $Y/M\cong FY/F\cong \langle H,\delta\rangle\cong \S_n$.
    Therefore, the normal quotient graph $\Gamma_M$ has $n$ vertices with valency $n-1$. Consequently, $\Gamma_M\cong\K_n$, and thus $\Gamma$ is a $2$-arc-transitive cover of $\K_n$.

    Finally, we show that $M=Y\cap F\cong T^m$ for some positive integer $m$.
    Let
    \[s=(g(23))^3=(f(132))^3=ff^{(123)}f^{(132)}\in M=Y\cap F.\]
    Note that, for any $\alpha\in \mathcal{A}$, we have that
    \[s(\alpha)=f(\alpha) f(\alpha^{(132)})f(\alpha^{(123)}).\]
    Let $\alpha=(123...n)\in O_1$.
    It is a routine to verify that
    \[{\alpha}^{(132)}\in O_1,\ {\alpha}^{(123)}\in O_{n-2},\ {\alpha}^{-1}\in O_{n-1},\ {(\alpha^{-1})}^{(132)}\in O_{n-1}\mbox{ and }{(\alpha^{-1})}^{(123)}\in O_{2}.\]
    Hence, we have that
    \begin{align*}
        s(\alpha) &=f(\alpha) f({\alpha}^{(132)})f({\alpha}^{(123)})=y^2x\ \mbox{ and }\\
        s(\alpha^{-1}) &=f(\alpha^{-1}) f({(\alpha^{-1})}^{(132)})f({(\alpha^{-1})}^{(123)}) =y^{-2}x.
    \end{align*}
    Since $s(\alpha)s(\alpha^{-1})^{-1}=y^2x(y^{-2}x)^{-1}=y^4$ and $y$ has odd prime order, it follows that $y\in \langle s(\alpha),s(\alpha^{-1})\rangle$.
    Consequently,
    \[\langle s(\alpha),s(\alpha^{-1})\rangle=\l y^{-2}x, y^2x\r=\langle y,x\rangle=T.\]
    Recall that $H$ acts transitively on $\mathcal{A}$.
    Choose $h\in H$ such that $\alpha^h=\alpha^{-1}$.
    Then both $s$ and $s^{h^{-1}}$ lie in $M=Y\cap F$, with $s^{h^{-1}}(\alpha)=s(\alpha^h)=s(\alpha^{-1})$.
    Since $\langle s(\alpha),s(\alpha^{-1})\rangle=T$, the group $Y\cap F$ projects onto at least one direct product component of $F\cong T^{|\mathcal{A}|}$.
    Further, since $Y/M\cong \S_n$ is transitive on the $|\mathcal{A}|$ components of $F$, $M$ projects onto every component.
    According to Lemma~\ref{lem:subdirect}, $M=Y\cap F\cong T^{m}$ is minimal in $Y$, where $m$ is the size of a $\S_n$-invariant partition on $\calA$ and is a divisor of $|\calA|=(n-1)!$.
    Hence, $Y\cong T^m.\S_n$ and $\Gamma$ is a minimal cover of $\K_n$ with transformation group $M$.

    Assume $n\geqslant 7$.
    Let $\beta=(142536\dots n)\in O_2$.
    Then $\beta^{(123)}=(243516\dots n)\in O_{n-4}$ and $\beta^{(132)}=(341526\dots n)\in O_{2}$.
    This gives
    \[s(\beta)=f(\beta)f(\beta^{(132)})f(\beta^{(123)})=x^2=1,\]
    implying that $M=Y\cap F$ is not a full-diagonal subgroup of $F$, and so $m>1$.
    Let $I_1$ be the block that contains $\alpha$ and let $D$ be the block stabilizer of $I_1$ in $\S_n$. It follows that $m=|\S_n:D|$.
    Since the conjugation action of $\alpha$  stabilizes itself, we have $\alpha \in D$ and $D$ is transitive under the natural action on $n$ points.
    By~\cite[Theorem 5.2B]{snmaximal}, if $|\S_n:D|<\binom{n}{\lfloor \frac{n}{2}\rfloor}$ when $n\geqslant 7$, then either
    \begin{enumerate}[\rm(1)]
        \item $(n,D,|\S_n:D|)=(2n_0,\S_{n_0}\wr\S_2,\frac{1}{2}\binom{2n_0}{n_0})$; or
        \item $(n,D,|\S_n:D|)=(7,\PSL(3,2),30)$ or $(8,\mathrm{AGL}(3,2),30)$.
    \end{enumerate}
    Note that $30>\frac{1}{2}\binom{7}{3}=\frac{35}{2}$ and $\mathrm{AGL}(3,2)$ does not contain any $n$-cycles by~\cite[Theorem 1.2]{cycles}.
    Thus, we have that $m=|\S_n:D|\geqslant \frac{1}{2}\binom{n}{\lfloor\frac{n}{2}\rfloor}$ when $n\geqslant 7$.
\end{proof}

\section{For the case \texorpdfstring{$n=4$}{n=4}}\label{sec:n4}

In this section, we focus on Construction~\ref{cons:simplecomplete} in the minimal case $n=4$.
We rewrite the construction as follows.

Let $\mathcal{A}$ be the set of all $4$-cycles of $\{1,2,3,4\}$, and write $\mathcal{A}=\{\alpha_1,\dots,\alpha_6\}$, where
\[ 
\alpha_1=(1234),\ \alpha_2=(1432),\ \alpha_3=(1243),\ 
\alpha_4=(1342),\ \alpha_5=(1324),\ \alpha_6=(1423).
\]
Recall that $H=\langle h_1,h_2\rangle\leqslant\S_4$ and $\delta=(12)$, where $h_1=(234)$ and $h_2=(34)$.  
Identifying the six elements of $\mathcal{A}$ with $\{1,\dots,6\}$, the permutations $h_1$, $h_2$, and $\delta$ acting on $\mathcal{A}$ with 
\[
h_1=(146)(235),\qquad 
h_2=(13)(24)(56),\qquad 
\delta=(14)(23)(56).
\]
Similarly, the map $f$ defined in Construction~\ref{cons:simplecomplete} can be represented as 
\[f=(y, y^{-1}, y, y^{-1}, x, x).\]

\begin{construction}\label{con:K4}
    Let $T=\langle x,y\rangle$ be a non-abelian simple group with $|x|=2$ and $|y|$ an odd prime, and let $X=T\wr \S_6= T^6{:}\S_6$ such that 
    \[(t_1,t_2,...,t_6)^\sigma=(t_{1^{\sigma^{-1}}},t_{2^{\sigma^{-1}}},...,t_{6^{\sigma^{-1}}})\mbox{ for $(t_1,t_2,...,t_6)\in T^6$ and $\sigma\in \S_6$.}\]
    Set $H=\langle h_1,h_2\rangle <\S_6$ and $g=f\delta\in X$.
    Let $Y=\langle H,g\rangle=\langle h_1,h_2,f\delta\rangle$.
    Define $\Gamma=\Cos(Y,H,HgH)$.\qed
\end{construction}

In the rest of this section, we always use the definitions given in Construction~\ref{con:K4}. 
By Lemma~\ref{lem:simplecomplete}, $\Gamma$ is a connected $(Y,2)$-arc-transitive cubic graph with $\Gamma_N\cong \K_4$ and $Y\cong T^d.\S_4$. 
However, it seems difficult to determine the exact value of $d$ for general cases of Construction~\ref{cons:simplecomplete}. The following theorem gives a complete description for $\K_4$ in different situations. 
For convenience, we define the following two sets:
\[\begin{aligned}
    \Phi_1&=\{\varphi\in \Aut(T) : \varphi(x)=x,\ \varphi(y)=y^{-1}\}\mbox{ and }\\
    \Phi_2&=\{\varphi\in \Aut(T) : \varphi(yxy)=y^2x,\ \varphi(y^2x)=yxy,\ \varphi(xy^2)=y^{-2}x\}.
\end{aligned}\]
Note that each of $\Phi_1$ and $\Phi_2$ contains at most one automorphism of $T$.
\begin{theorem}\label{thm:k4}
    The group $Y\cong T^d.\S_4$ has a minimal normal subgroup $M\cong T^d$, and $\Gamma$ is a connected $(Y,2)$-arc-transitive $T^d$-cover of $\K_4$.
    Moreover, $d\in\{1,3,6\}$ and exactly one of the following statements holds.
        \begin{enumerate}[\rm(1)]
            \item $d=1$ if and only if both $\Phi_1$ and $\Phi_2$ are non-empty;
            \item $d=3$ if and only if $\Phi_1$ is non-empty and $\Phi_2$ is empty;
            \item $d=6$ if and only if $\Phi_1$ is empty.
        \end{enumerate}
\end{theorem}

Before proving Theorem~\ref{thm:k4}, we give the following lemma which will be used in the proof.
\begin{lemma}\label{lem:k4}
    Let $N=Y\cap T^6$.
    \begin{enumerate}[\rm(1)]
        \item Let $C\cong N.\bbZ_2^2$ be the full-preimage of $\bbZ_2^2\lhd \S_4\cong Y/N$.
        Then $\Gamma\cong \mathrm{Cay}(C,S)$ with $S=\{ s_1,s_2,s_3\}$such that 
        \[\begin{array}{l}
            s_1=gh_2=(y,y^{-1},y,y^{-1},x,x)(12)(34),\\
            s_2=h_2 h_1 g h_1=(x,x,y^{-1},y,y,y^{-1})(34)(56)\\
            s_3=h_2 h_1^{-1} g h_1^{-1}=(y^{-1},y,x,x,y^{-1},y)(12)(56).
        \end{array}
        \]
        \item $N=\langle t_1,t_2,t_3\rangle$, where
        \[\begin{array}{lllllll}
            t_1=(yxy&,y^{-1}xy^{-1}&,y^2x&,y^{-2}x&,xy^2&,xy^{-2}&),\\
            t_2=(y^2x&,y^{-2}x&,yxy&,y^{-1}xy^{-1}&,xy^{-2}&,xy^2&),\\
            t_3=(xy^{2}&,xy^{-2}&,y^{-2}x&,y^2x&,yxy&,y^{-1}xy^{-1}&).
        \end{array}
        \]
    \end{enumerate}
\end{lemma}
\begin{proof}
    (1).
    Note that $C/N\cong \bbZ_2^2$ is regular on $\K_4$.
    Then $C$ is regular on $\Gamma$.
    This yields that $\Gamma$ is isomorphic to a Cayley graph on $C$.
    It is not hard to see that $HgH=Hs_1H=HSH$ and $s_1,s_2,s_3\in C$.
    We deduce that $\Gamma=\Cos(Y,H,HgH)\cong \Cos(C,1,\{s_1,s_2,s_3\})\cong \mathrm{Cay}(C,S)$.   

    (2).
    Let $R=S\cup\{\mathrm{Id}(Y)\}$.
    Then $R$ is the right transversal for $C$ mod $N$.
    Let $\psi: Y\rightarrow R$ be a map such that $Ny=N\psi(y)$ with $\psi(y)\in R$ for any $y\in Y$.
    Note that $s_1,s_2,s_3$ are involutions.
    By \cite[Lemma 4.2.1]{seress2003Permutation}, we have that
    \[\begin{array}{ll}
        N&=\langle rs\psi(rs)^{-1}:r\in R,s\in S\rangle=\langle s_is_j\psi(s_is_j)^{-1}: i,j\in\{1,2,3\}\rangle\\
        &=\langle s_is_js_k:\{i,j,k\}=\{1,2,3\}\rangle=\langle s_1s_2s_3,s_1s_3s_2,s_2s_1s_3\rangle.
    \end{array}\]
    Let $t_1=s_1s_2s_3$, $t_2=s_1s_3s_2$ and $t_3=s_2s_1s_3$.
    Then 
    \[\begin{aligned}
        t_1&=s_1s_2s_3=(gh_2)(h_2h_1gh_1)(h_2 h_1^{-1} g h_1^{-1})=ff^{(\delta h_1)^{-1}}f^{(\delta h_1\delta h_1h_2h_1^{-1})^{-1}}\\
        &=ff^{(1526)}f^{(3546)}=(yxy,y^{-1}xy^{-1},y^2x,y^{-2}x,xy^2,xy^{-2}),\\
        t_2&=s_1s_3s_2=(gh_2)(h_2 h_1^{-1} g h_1^{-1})(h_2h_1gh_1)=ff^{(\delta h_1^{-1})^{-1}}f^{(\delta h_1^{-1}\delta h_1^{-1}h_2h_1)^{-1}}\\
        &=ff^{(3645)}f^{(1625)}=(y^2x,y^{-2}x,yxy,y^{-1}xy^{-1},xy^{-2},xy^2),\\
        t_3&=s_2s_1s_3=(h_2h_1gh_1)(gh_2)(h_2 h_1^{-1} g h_1^{-1})=f^{(h_2h_1)^{-1}}f^{(h_2h_1\delta h_1)^{-1}}f^{(h_2h_1\delta h_1\delta h_1^{-1})^{-1}}\\
        &=f^{(15)(26)(34)}f^{(34)(56)}f^{(3546)}=(xy^{2},xy^{-2},y^{-2}x,y^2x,yxy,y^{-1}xy^{-1}).
    \end{aligned}\]
    The proof is complete.
\end{proof}

Now we are ready to prove Theorem~\ref{thm:k4}.
\begin{proof}[Proof of Theorem~\ref{thm:k4}.]
    By Lemma~\ref{lem:k4}, we have that $\Gamma$ is a connected $(Y,2)$-arc-transitive $N$-cover of $\K_4$, where $N=Y\cap T^6=\langle t_1,t_2,t _3\rangle$ with 
    \[\begin{array}{lllllll}
        t_1=(yxy&,y^{-1}xy^{-1}&,y^2x&,y^{-2}x&,xy^2&,xy^{-2}&),\\
        t_2=(y^2x&,y^{-2}x&,yxy&,y^{-1}xy^{-1}&,xy^{-2}&,xy^2&),\\
        t_3=(xy^{2}&,xy^{-2}&,y^{-2}x&,y^2x&,yxy&,y^{-1}xy^{-1}&).
    \end{array}
    \]
    where $T=\langle x,y\rangle$ with $|x|=2$ and $|y|$ an odd prime.
    It is not hard to see that $N$ projects onto each direct product component.
    Hence, we have that $N\cong T^d$ for some $1\leqslant d\leqslant 6$.
    For convenience, let $\rho_i$ be the projection map from $T^6$ to the $i$-th direct product components for each $i=1,...,6$.

    \textbf{Step 1.} Let $i,j\in\{1,...,6\}$ with $i\neq j$. 
    We prove that the following statements are equivalent: 
    \begin{enumerate}[\rm(i)]
        \item  there exists $\varphi\in \Aut(T)$ such that $\varphi(\rho_i(z))=\rho_j(z)$ for every $z\in \{t_1,t_2,t_3\}$;
        \item there exists $\varphi\in \Aut(T)$ such that $\varphi(\rho_i(z))=\rho_j(z)$ for every $z\in N$;
        \item for any $h\in\langle h_1,h_2,\delta\rangle\cong \S_4$, there exists $\psi\in\Aut(T)$ such that $\psi(\rho_{i^h}(z))=\rho_{j^h}(z)$ for every $z\in N$.
    \end{enumerate}

    Clearly, part~(iii) implies part~(ii) and part~(ii) implies part~(i).
    Conversely, since $\l t_1,t_2,t_3\r=N$, part~(i) implies part~(ii). 
    Now, we assume that part~(ii) holds. 
    Lemma~\ref{lem:simplecomplete} implies that $Y\cong T^d.\S_4$, and then $T^6Y=T^6{:}\l h_1, h_2, \delta\r$. 
    Hence, for any $h\in \langle h_1,h_2,\delta\rangle$, there exist $\lambda\in Y$ and $\mu=(u_1,\dots, u_6)\in T^6$ such that $h=\mu^{-1}\lambda$. 
    Now, let $\tau_{i}$ and $\tau_j$ be the inner automorphisms of $T$ induced by $u_i$ and $u_j$, respectively, and set $\psi=\tau_j\circ \varphi\circ \tau_i^{-1}$. 
    Since  $N\lhd Y$, we have $N=N^{\lambda}=N^{\mu h}$. Hence, for any $z\in N$, there exists $w\in N$ such that $z=w^{\mu h}$ and for any $m\in \{1,\dots, 6\}$, 
    \[\rho_{m^h}(z) =\rho_{m^h}(w^{\mu h}) =\rho_m(w^{\mu } )=\left(\rho_m(w)\right)^{u_m}.\]
    It follows that
    \[\begin{aligned}
        \psi(\rho_{i^h}(z))&=\psi\bigl(\rho_{i}(w)^{u_i}\bigr)=\psi\circ\tau_i(\rho_i(\omega))=(\tau_j\circ \varphi) (\rho_{i}(w))&\\
        &=\tau_j (\rho_{j}(w))=\rho_j(w)^{u_j}=\rho_{j^h}(z).
    \end{aligned}\] 
    Hence, $\psi$ is the desired automorphism of $T$ and part~(iii) holds.

    \textbf{Step 2.} We show that $d<6$  if and only if there exists $\varphi\in \Phi_1$.

    By Lemma~\ref{lem:subdirect}, there is a block system $I=\{I_1,...,I_m\}$ of $\langle h_1,h_2,\delta\rangle<\S_6$ such that $N=\prod_{j=1}^m D_j$, where $D_j\cong T$ is a full-diagonal subgroup of $\prod_{i\in I_j}T_i$.
    Note that $\langle h_1,h_2,\delta\rangle\cong\S_4$ has exactly one non-trivial proper block system:
    \[\{\{1,2\},\{3,4\},\{5,6\}\}.\] 
    If $d<6$, then without loss of generality we may assume $\{1,2\}\subseteq I_1$. 
    Then there exists $\varphi\in \Aut(T)$ such that $\varphi(\rho_1(z))=\rho_2(z)$ holds for all $z\in N$. 
    This gives 
    \[\varphi(y^4)=\varphi(\rho_1(t_2t_3))=\rho_2(t_2t_3)=y^{-4}.\]
    Since $y$ is of odd order, we have $\varphi(y)=y^{-1}$ and 
    \[\varphi(x)=\varphi(\rho_1(t_3) y^{-2})=\varphi(\rho_1(t_3) )y^{2}=\rho_2(t_3) y^{2} =x. \]
    Hence $\varphi\in \Phi_1$.
    Conversely, if $\varphi\in \Phi_1$, then it is straightforward to see that \[\varphi(\rho_1(t_i))=\rho_2(t_i) \mbox{ for $i\in \{1,2,3\}.$}\] 
    Since $N=\l t_1,t_2,t_3\r$, we have that $N$ is a proper subdirect subgroup of $T^6$ which gives $d<6$. 
    
    \textbf{Step 3.} Finally, we show that $d=1$ if and only if both $\Phi_1$ and $\Phi_2$ are non-empty.
    
    The proof in Step 2 shows that $\Phi_1\neq \emptyset$ if and only if $I=\{\{1,2\},\{3,4\},\{5,6\}\}$ or $\{\{ 1,\dots, 6\}\}$. 
    Hence, it suffices to show that, under the assumption $\Phi_1$ is non-empty, $\Phi_2\neq \emptyset$ if and only if $I=\{\{ 1,\dots, 6\}\}$. 
    
    If $I=\{\{ 1,\dots, 6\}\}$, then there exists $\varphi_2\in \Aut(T)$ such that $\varphi_2(\rho_1(t_i))=\rho_3(t_i)$ holds for $i\in \{1,2,3\}$. 
    Obviously, $\varphi_2\in \Phi_2$ which gives $\Phi_2$ is non-empty.
    Conversely, if $\Phi_2$ is non-empty, then it is straightforward that $1$ and $3$ belong to some same part of the partition $I$, and thus $I=\{\{1,\dots, 6\}\}$.
    We complete the proof. 
\end{proof}

As an application of Theorem~\ref{thm:k4}, the following example shows that not all simple groups can occur for the possible values of $d$.
\begin{example}
    \begin{enumerate}[\rm(1)]
        \item It is proved in~\cite{bredadazevedo2021Strong} and~\cite{leemans2017Chiral} that a non-abelian simple group $T$ has the property that, for every generating pair $(x,y)$ with $|x|=2$, there always exists an automorphism centralizing $x$ and inverting $y$ if and only if
        \[T\in\{\A_7,\ \PSL_2(q),\ \PSL_3(q),\ \PSU_3(q)\}.\]
        By Theorem~\ref{thm:k4}, if $T$ is one of $\A_7$, $\PSL_2(q)$, $\PSL_3(q)$ and $\PSU_3(q)$, then $d\in\{1,3\}$. 
        Let $T=\A_5\cong\PSL_2(5)$.
        \begin{enumerate}[\rm(i)]
            \item Choose $x=(12)(34)$ and $y=(12345)$.
            Let $b=(14)(35)$. 
            Elementary calculations show that
            \[ (yxy,y^2x,xy^2)^b=(y^2x,yxy,y^{-2}x)=\bigl((14235),(12534),(13254)\bigr).\]
            Theorem~\ref{thm:k4} implies that $Y\cap T^6\cong T$, and hence $Y\cong \A_5.\S_4$.
            By Magma, we have that $Y\cong \A_5\times \S_4$.
            In this case, $\Gamma$ is an $\S_4$-cover of the Petersen graph.

            \item Choose $x=(12)(34)$ and $y=(153)$.
            Then $y^{4}=y=(153)$ and $[y^{-1},x]=(12453)$.
            Thus, there is no automorphism that maps $y^{4}$ to $[y^{-1},x]$.
            Note that
            \[y^4=(y^2x)(xy^2)\mbox{ and }[y^{-1},x]=(yxy)(y^{-2}x).\]
            This implies that there is no automorphism that maps $(y^2x,xy^2)$ to $(yxy,y^{-2}x)$.
            Theorem~\ref{thm:k4} yields that $Y\cap T^6\cong T^3$ and $Y\cong \A_5^3.\S_4$.
            By Magma, we have that $Y\cong (\A_5^3\times \bbZ_2^2){:}\S_3$.
            Note that $\Gamma$ is a normal $2$-arc-transitive Cayley graph on $\A_5^3\times \bbZ_2^2$.
        \end{enumerate}

        \item For prime $p\geqslant 11$, let $T=\A_p$, $x=(12)(36)$ and $y=(12\cdots p)$.
        Then there exists no element in $\S_p$ which centralizes $x$ and inverts $y$.
        Thus, by choosing such triple $(T,x,y)$, the group $Y$ in Construction~\ref{con:K4} equals to $T^6{:}\S_4$.
    \end{enumerate}
\end{example}

\vskip0.2in
\noindent\thanks{{\bf Acknowledgments.}
The authors are grateful to the anonymous referees for their valuable comments and suggestions that have helped to improve the paper.

This work was supported by the National Natural Science Foundation of China (No. 11931005, 12101518, 11901512), the Fundamental Research Funds for the Central Universities (No. 20720210036, 20720240136), and the Natural Science Foundation of Yunnan Province (No. 202401AT070275). 
}

\end{document}